\documentclass[12pt,oneside,british]{amsart}
\usepackage{xspace}
\usepackage{graphicx}
\usepackage{mathrsfs}
\usepackage{latexsym}
\usepackage{amsfonts}
\usepackage{amsmath}
\usepackage{amscd}
\usepackage{amssymb}
\usepackage{amsthm}
\usepackage{enumerate}
\usepackage{changepage}
\usepackage{hyperref,color}
\hypersetup{
	colorlinks,
	citecolor=black,
	filecolor=black,
	linkcolor=black,
	urlcolor=black
}
\input{xy}
\xyoption{all}
\xyoption{curve}

\usepackage[T1]{fontenc}
\usepackage[latin9]{inputenc}
\usepackage{textcomp}
\usepackage{amsthm}
\usepackage{amsmath}
\usepackage{amssymb}
\usepackage{esint}
\usepackage[all]{xy}

\makeatletter
\newcommand{\xyR}[1]{
	\xydef@\xymatrixrowsep@{#1}}
\newcommand{\xyC}[1]{
	\xydef@\xymatrixcolsep@{#1}}

\makeatother

\newcommand{\bK}{{\mathbb K}}

\newcommand{\bP}{{\mathbb P}}

\newcommand{\bZ}{{\mathbb Z}}

\newcommand{\bG}{{\mathbb G}}

\newcommand{\cE}{{\mathcal E}}
\newcommand{\cF}{{\mathcal F}}

\newcommand{\cO}{{\mathcal O}}
\newcommand{\cT}{{\mathcal T}}
\newcommand{\cM}{{\mathcal M}}
\newcommand{\cN}{{\mathcal N}}

\newcommand{\tM}{{\widetilde{M}}}

\newcommand{\cExt}{{\underline{\mathcal{E}xt}}}

\newcommand{\Gr}{{{\mbox{\rm --\,GrMod}}}}

\newcommand{\Hom}{{{\mbox{\rm Hom}}}}
\newcommand{\h}{{{\mbox{\rm h}}}}

\newcommand{\Qc}{{{\mbox{\rm --\,Qcoh}}}} 
\newcommand{\Qcoh}{{{\mbox{\rm Qcoh}}}}

\newcommand{\Tor}{{{\mbox{\rm --\,Tors}}}}

\newcommand{\Sat}{{{\mbox{\rm Sat}}}}
\newcommand{\Sa}{{{\mbox{\rm --\,Sat}}}} 
\newcommand{\Proj}{ {{\mbox{\rm Proj}}}}

\newcommand{\Ext}{{{\mbox{\rm Ext}}}}

\newcommand{\Sym}{{{\mbox{\rm Sym}}}}

\newcommand{\lcm}{{{\mbox{\rm lcm}}}}
\newcommand{\Sh}{{{\mbox{\rm Sh}}}}

\newcommand{\ve}{\ensuremath{\mathbf{e}}\xspace}

\newcommand{\vm}{\ensuremath{\mathbf{m}}\xspace}
\newcommand{\vv}{\ensuremath{\mathbf{v}}\xspace}

\newcommand{\vx}{\ensuremath{\mathbf{x}}\xspace}

	\newtheorem{theorem}{Theorem}
	\newtheorem{prop}[theorem]{Proposition}
	\newtheorem{lemma}[theorem]{Lemma}
	
	\newtheorem{cor}[theorem]{Corollary}
	
	\newtheorem{defn}[theorem]{Definition}
	
	\newtheorem{conj}[theorem]{Conjecture}
	\theoremstyle{plain}
	
	\theoremstyle{plain}
	
\begin{document}
		\title{Ample tangent bundle on smooth projective stacks}
		\author{Karim El Haloui}
		\email{K.El-Haloui@warwick.ac.uk}
		\address{Dept. of Maths, University of Warwick, Coventry,
			CV4 7AL, UK}
		\date{Novembre 06, 2016}
		\subjclass[2010]{Primary 14A20; Secondary 18E40}
		
		\begin{abstract}
			We study ample vector bundles on smooth projective stacks. In particular, we prove that the tangent bundle for the weighted projective stack $\bP(a_0,...,a_n)$ is ample. A result of Mori shows that the only smooth projective varieties with ample vector bundle are isomorphic to $\bP^N$ for some $N$. Extending our geometric spaces from varieties to projective stacks, we are able to provide a new example. This leaves the open question of knowing if the only smooth projective stacks are the weighted projective stacks. 
		\end{abstract}
		
		\maketitle
		
		A conjecture of Hartshorne says that the only $n$-dimensional irreducible smooth projective space whose tangent bundle is ample is isomorphic to $\bP^N$ for some $N$. This was proved for all dimensions and for any characteristic of the base field by Mori \cite{Mor}.
		
		The goal of this paper is to extend the definition of ample vector bundles to smooth projective stacks and provide a new example in this context. We prove that for any positive weights $a_0,...,a_n$, the tangent bundle of the weighted projective stack $\bP(a_0,...,a_n)$ is ample. 
		The next step would be to show whether the only spaces with such characterisation are the weighted projective stacks. We leave this as a conjecture in this article.
		
		In section 1 we recall some general observations about quotient categories.
		In section 2 we establish a technical framework for working with ample bundles
		on smooth projective stacks.
		In section 3 we use this framework to prove that for the tangent bundle of any weighted projective stack is ample.
		\section*{Acknowledgement}
		The author would like to thank 
		Dmitriy Rumynin
		for his constant help and support.
		
		\section{Properties of the quotient category}
		
		Let $\bK$ be a field. All our modules are graded modules and homomorphisms are graded module homomorphisms. 
		\subsection{Quotient stack}
		
		Let $Y$ be a smooth algebraic variety with an action of an algebraic group $G$. $\cO_{[X]}$-modules on the quotient stack $[X]=[Y/G]$ can be understood in terms of $G$-equivariant $\cO_Y$-modules \cite{EHR}.
		
		There are different notions of a projective stack, for instance,
		a stack whose coarse moduli space is a projective variety.
		Here we use a more restrictive notion:
		a projective stack is a smooth closed substack of a weighted
		projective stack \cite{Zho}. Let us spell it out.
		Let $V=\oplus_{k=1}^m V_k$ be a positively graded 
		$n+1$-dimensional $\bK$-vector space. Naturally
		we treat it as a $\bG_m$-module with positive weights by
		$\lambda \bullet \vv_k = \lambda^k \vv_k$ where $\vv_k \in V_k$.
		Let $Y$ be a smooth closed $\bG_m$-invariant subvariety of $V\setminus
		\{0\}$. We define {\em a projective stack} as the stack $[X]=[Y/\bG_m]$.
		The G.I.T.-quotient $X=Y//\bG_m$ is the coarse moduli space of $[X]$.
		
		\subsection{Coherent sheaves on projective stacks}
		
		Let us describe the category of $\cO_{[X]}\Qc$ of quasicoherent sheaves on $[X]$.
		Choose a homogeneous basis $\ve_i$ on $V$ 
		with $\ve_i\in V_{d_i}$,
		$i=0,1, \ldots, n$.
		Let $\vx_i \in V^\ast$ be the dual basis. Then 
		$\bK[V] = \bK [\vx_0,...,\vx_n ]$ possesses a natural grading
		with  $\deg(\vx_{i})=d_{i}$.
		Let $I$ be the defining ideal of $\overline{Y}$. Since $Y$ is $\bG_m$-invariant,
		the ideal $I$ and the ring
		$$
		A:=\bK[Y]= \bK[\overline{Y}] = \bK [\vx_0,...,\vx_n ]/I
		$$
		are graded.
		Both $X$ and $[X]$ can be thought of as the projective spectrum of
		$A$.
		The scheme $X$ is naturally isomorphic to the scheme theoretic 
		$\Proj\, A$.
		The stack $[X]$ is the Artin-Zhang projective spectrum
		$\Proj_{AZ} A$ \cite{AKO}, 
		i.e. its category of quasicoherent sheaves
		$\cO_{[X]}\Qc$ is equivalent to the quotient category
		$A\Gr/A\Tor$ where
		$A\Gr$ is the category of
		$\bZ$-graded $A$-modules, 
		$A\Tor$ is its full subcategory of torsion modules (we
		identify the objects of $\cO_{[X]}\Qc$ and $A\Gr/A\Tor$).
		
		Recall that 
		$$
		\tau (M) = \{\vm \in M \,\mid\, \exists N \; \forall k>N \; A_k\vm=0\}
		$$
		is {\em the torsion submodule of} $M$. $M$ is said to be {\em torsion} 
		if $\tau (M)=M$. It can be seen as well that the torsion submodule of $M$ is the sum of all the finite dimensional submodules of $M$ since $A$ is connected.
		
		Denote by
		$$
		\pi_A:A\Gr \rightarrow A\Gr/A\Tor
		$$
		the quotient functor. Since $A\Gr$ has enough injectives and $A\Tor$ is dense then there exists a section functor 
		$$
		\omega_A:A\Gr/A\Tor \rightarrow A\Gr
		$$ which is right adjoint to $\pi_A$ in the sense that 
		$$
		\Hom_{A\mbox{\tiny \Gr}}(N,\omega_A(\mathcal{M}))\cong\Hom_{A\mbox{\tiny\Gr}/A\mbox{\tiny\Tor}}(\pi_A(N),\mathcal{M}).
		$$
		Recall as well that $\pi_A$ is exact and that $\omega_A$ is left exact.
		We call $\omega_A\pi_A(M)$ the {\em $A$-saturation} of $M$. We say that a module is {\em $A$-saturated} is it is isomorphic to the saturation of a module. It can be easily seen that a saturated module is the saturation of itself and is torsion-free (from the adjunction). If $M$ and $N$ are $A$-saturated, then being isomorphic in $A\Gr/A\Tor$ is equivalent to being isomorphic in $A\Gr$.
		
		The $\cO_{[X]}$-module $\cO_{[X]}(k)$ is defined
		as $\Sat(A[k])$ where $A[k]$ is the shifted regular module and the grading is given by $A[k]_m = A_{k+m}$.
		
		In particular, $A[k]$ is $A$-saturated if $A$ is a polynomial rings of more than two variables \cite{AZ}. A well-known example of a ring which isn't $A$-saturated is the polynomial ring of one variable $A=\bK[x]$ which $A$-saturation is given by $\bK[x,x^{-1}]$.
		
		\subsection{Tensor product}
		
		Let $M$ and $N$ be two $A$-modules, then $M\otimes_A N$ possesses a natural $A$-module structure. We want to induce on the quotient category $A\Gr/A\Tor$ a structure of a symmetric monoidal category. Consider the full subcategory of $A$-saturated modules $A\Sa$. The essential image of the section functor $\omega_A$ consists precisely of the $A$-saturated modules. Thus $$\omega_A:  A\Gr/A\Tor \rightarrow  A\Gr$$ becomes full and faithful onto its image. Indeed
		$$
		\Hom_{A\mbox{\tiny\Gr}/A\mbox{\tiny\Tor}}(\pi_A(M),\pi_A(N)) \cong \Hom_{A\mbox{\tiny\Gr}}(M,N) 
		$$
		if $N$ is $A$-saturated. So now, we identify the quotient category with its image and call its objects sheaves which we denote by curly letters. If $N$ is a finitely generated $A$-module, then $\cN=\pi_A(N)$ is said to be a \emph{coherent} $\cO_{[X]}$-module. The definition makes sense since the $A$-saturation of a finitely generated $A$-module $N$ is finitely generated. Indeed, It was proved \cite{AZ} that for any graded $A$-module $N$ we have:
		$$
		0  \rightarrow \tau_A(N) \rightarrow N \rightarrow \omega_A\pi_A(N) \rightarrow R^{1}\tau_A(N) \rightarrow 0.
		$$
		where $\tau_A(N)$ and $R^{1}\tau_A(N)$ are torsion $A$-modules. Since the localisation functor is exact and that the localisation of a torsion module is zero, then the localisation of the saturation of $N$ is isomorphic to the localisation of $N$ which in turn is finitely generated. But being finitely generated is a local property, thus the saturation of $N$ is finitely generated.
		From general localisation theory \cite{Ga}, $A\Sa$ is an abelian category (it is actually a Grothendieck category) but it is not an abelian subcategory of $A\Gr$
		\footnote{A full subcategory of an abelian category need not be an abelian subcategory}. The kernels in both categories are the same but the cokernel of two saturated $A$-modules isn't necessarily saturated.
		The saturation functor $\Sat:A\Gr \rightarrow A\Sa$ is exact and its right adjoint, namely the inclusion functor, is left exact. Moreover it preserves finite direct sums as does any additive functor in any additive category.
		
		Let $A$ be the coordinate ring of some projective stack $[X]$.  The graded global section functor $\Gamma_*:\Qcoh([X]) \rightarrow A\Gr$ and the sheafification functor $\Sh:A\Gr \rightarrow \Qcoh([X])$ induces the following equivalence of categories:
		$$
		A\Sa \cong \Qcoh([X]).
		$$
		
		There exists a symmetric monoidal structure in the category of quasicoherent sheaves on $[X]$ denoted by $\Qcoh([X])$. The latter category is equivalent to $A\Gr/A\Tor$ and hence to $A\Sa$. Note that $\Sh({M}) \otimes \Sh({N}) \cong \Sh({M \otimes_A N})$ in $\Qcoh([X])$. So,
		
		\begin{align*}
		\Gamma_*(\Sh({M}) \otimes \Sh({{N}}))
		&\cong \Gamma_*(\Sh{(M \otimes_A N})) \\
		&\cong \Sat(M \otimes_A N)
		\end{align*}
		where all the isomorphisms are natural (to preserve the symmetric monoidal structure).
		
		We can now define a tensor product: take $\cM$ and $\cN$ in $A\Sa$ and let
		$$
		\cM \otimes \cN := \Sat(M\otimes_A N).
		$$
		where as objects $\cM=\Sat(M)$ and $\cN=\Sat(N)$.
		Since $\Sat$ and the tensor product of graded modules is right-exact so is the tensor product defined on $A\Sa$.
		
		\section{Ample vector bundles}
		
		We want to define a notion of vector bundles of finite rank that we shall call equivalently locally free sheaf of finite rank which will be defined purely in cohomological terms.
		
		We have a \emph{internal} Hom defined on $A\Sa$ defined as follow: 
		$$
		\underline{\Hom}_{\mbox{\tiny A\Sa}}(\cM,\cN)= \bigoplus_{k \in \bZ} \Hom_{\mbox{\tiny A\Sa}}(\cM,\cN[k])
		$$
		where as objects $\cN[k]=\Sat(N[k])$ (saturation is preserved under shifts).
		
		The injective objects in $A\Sa$ are the injective torsion-free $A$-modules in $A\Gr$ (they are all saturated) and from standard localisation theory $A\Sa$ has enough injectives \cite{Ga}. Moreover an injective object in $A\Gr$ can be decomposed as a direct sum of an injective torsion-free $A$-module and an injective torsion  $A$-modules determined up to isomorphism \cite{AZ}. So the injective resolution of a $A$-module $N$, say $E^\bullet(N)$, is equal to $Q^\bullet(N) \oplus I^\bullet(N)$ where $Q^\bullet(N)$ is the saturated torsion free part and $I^\bullet(N)$ the torsion free part. Supposing from now on that $M$ is a finitely generated graded $A$-module then:
		\begin{align*}
		\Ext^i_{\mbox{\tiny A\Sa}}(\cM,\cN) &= R^i\Hom_{\mbox{\tiny A\Sa}}(\cM,\_)(\cN) \\
		&\cong \h^i(\Hom_{\mbox{\tiny A\Gr}} (M,Q^\bullet(N)))
		\end{align*}
		Graded Ext is defined as follows:
		\begin{align*}
		\underline{\Ext}^i_{\mbox{\tiny A\Sa}}(\cM,\cN) &= \bigoplus_{k \in \bZ} \Ext^i_{\mbox{\tiny A\Sa}}(\cM,\cN[k]) \\
		&\cong \h^i(\underline{\Hom}_{\mbox{\tiny A\Gr}}(M,Q^\bullet(N)))
		\end{align*}
		which is the $i$th right derived functor of $\underline{\Hom}$.
		We can endow the graded Ext in $A\Sa$ with the structure of a graded $A$-module and define the sheafified version of graded Ext as follows
		$$
		\cExt^i(\cM,\cN) := \Sat(\underline{\Ext}^i_{\mbox{\tiny A\Sa}}(\cM,\cN))
		$$ 
		where $\cM$ and $\cN$ are objects in $A\Sa$ (and as object in $A\Gr$ $\pi_A(M)=\cM$ and $M$ is finitely generated).
		
		Let $X$ be a smooth projective variety and $\cE$ a vector bundle (of finite rank). Equally, $\cE$ is a locally free sheaf which is equivalent to asking that for all $x\in X$ the stalk $\cE_x$ is a free module of finite rank over the regular local ring $\cO_x$. But $\cE_x$ is a free module if and only if $\Ext^i_{\cO_x}(\cE_x,\cO_x)=0$ for all $i>0$. Since $\Ext^i_{\cO_x}(\cE_x,\cO_x)\cong \mathcal{E}xt^i_{\cO}(\cE,\cO)_x$ for all $x\in X$, then $\cE$ is a vector bundle if and only if $\mathcal{E}xt^i_{\cO}(\cE,\cO)=0$ for all $i>0$. This justify the next definition,
		
		\begin{defn}
			Let $\cM$ be a coherent sheaf. $\cM$ is a {\bf vector bundle} or {\bf a locally free sheaf} if
			$$
			\cExt^i(\cM,\cO)=0
			$$
			for all $i>0$ where $\cO:=\Sat(A)$.
		\end{defn}
		
		For example, if $[X]$ is a weighted projective stack of dimension greater than 2 then $A$ is a graded polynomial ring with more than 2 variables. In this case, it is known that $\cO=A$ \cite{AZ}. But since $\cO(k)$ is projective then $\cO(k)$ is locally free for all $k$.
		
		\begin{defn}
			A sheaf $\cM$ is said to be {\bf weighted globally generated} if there exists an epimorphism 
			$$
			\bigoplus_{j = 0}^{l-1} \cO(j)^{\oplus s_j} \rightarrow \cM \rightarrow 0
			$$
			for some non-negative $s_j$'s with $l=\lcm(d_0,...,d_n)$.
		\end{defn}
		
		In the case where all the weights are one, the least common multiple is equal to one and we recover the definition of globally generated sheaves adopted for projective varieties.
		
		\begin{prop}
			\begin{enumerate}
				\item Any quotient of a weighted globally generated sheaf is weighted globally generated.
				\item The direct sum of two weighted globally generated sheaves is weighted globally generated.
				\item For all $k\geqslant 0$, $\cO(k)$ is weighted globally generated.
				\item The tensor product of two weighted globally generated sheaves is weighted globally generated.
			\end{enumerate}
		\end{prop}
		
		\begin{proof}
			\begin{enumerate}
				\item It follows from the definition and the fact that the composition of two epimorphisms is an epimorphism.
				\item This follows immediately by definition.
				\item By the division algorithm, we know that $k=al+r$ for some non-negative integer $a$ and $0 \leqslant r < l$. We claim that the following map 
				$$
				\cO(r)^{\oplus (n+1)} \rightarrow \cO(k)
				$$
				induced by $(0,...,1_j,...,0) \mapsto x_j^{ \frac{al}{d_j}}$ is an epimorphism in $A\Sa$.
				To prove our claim we need to show that the cokernel of the map in $A\Gr$ is torsion. Take a homogeneous element $f\in A(k)$ and let $N=\max\left\lbrace  \dfrac{al}{d_j}, j\in\{0,...,n\} \right\rbrace $. Suppose that $h\in A$ is an homogeneous element of degree greater than $N$. So it can be written as $h'x_j^{ \frac{al}{d_j}}$ for some $j\in\{0,...,n\}$. It follows that 
				$hf$ is in the image of the map. Hence its cokernel is zero.
				\item Suppose that $\cM_1$ and $\cM_2$ are weighted globally generated. Then we know that 
				$$
				\bigoplus_{j = 0}^{l-1} \cO(j)^{\oplus s_j^1} \rightarrow \cM_1 \rightarrow 0 \:\:\:(1)
				$$
				and
				$$
				\bigoplus_{k = 0}^{l-1} \cO(k)^{\oplus s_k^2} \rightarrow \cM_2 \rightarrow 0 \:\:\:(2)
				$$
				Tensoring (2) by $\bigoplus_{j = 0}^{l-1} \cO(j)^{\oplus s_j^1}$ on the left and (1) by $\cM_2$ on the right, it follows that 
				$$
				\bigoplus_{j = 0}^{l-1} \cO(j)^{\oplus s_j^1} \otimes \bigoplus_{k = 0}^{l-1} \cO(k)^{\oplus s_k^2} \rightarrow \bigoplus_{j = 0}^{l-1} \cO(j)^{\oplus s_j^1} \otimes \cM_2 \rightarrow 0.
				$$
				and
				$$
				\bigoplus_{j = 0}^{l-1} \cO(j)^{\oplus s_j^1} \otimes \cM_2 \rightarrow \cM_1 \otimes \cM_2 \rightarrow 0
				$$
				Since the composition of epimorphisms is an epimorphism, we get
				$$
				\bigoplus_{j = 0}^{l-1} \cO(j)^{\oplus s_j^1} \otimes \bigoplus_{k = 0}^{l-1} \cO(k)^{\oplus s_k^2} \rightarrow \cM_1 \otimes \cM_2 \rightarrow 0.
				$$
				Therefore,
				$$
				\bigoplus_{0 \leqslant j+k \leqslant 2(l-1)} \cO(j+k)^{\oplus (s_j^1+s_k^2)} \rightarrow \cM_1 \otimes \cM_2 \rightarrow 0.
				$$
				Since each summand is weighted globally generated and that a direct sum of such is weighted globally generated, the result follows.
			\end{enumerate}
			
		\end{proof}

		In any abelian symmetric braided tensor category we can define the $n^{th}$ symmetric power functor $\Sym^n:A\Sa \rightarrow A\Sa$ as the coequalizer of all the endomorphisms $\sigma \in S_n$ of the $n^{th}$ tensor power functor $T^n$. Hence, we get the following well-known properties:
		
		\begin{prop}
			\begin{enumerate}
				\item There exists an epimorphism
				$$
				\Sym^p(\cM) \otimes \Sym^q(\cM) \twoheadrightarrow \Sym^{p+q}(\cM).
				$$
				\item There is a natural isomorphism
				$$
				\bigoplus_{p+q=n} \Sym^p(\cM) \otimes \Sym^q(\cN) \rightarrow \Sym^n(\cM\oplus \cN).
				$$
				\item The functor $\Sym^n$ preserves epimorphisms and sends coherent sheaves to coherent sheaves.
				\item There is a natural epimorphism
				$$
				\Sym^n(\cM) \otimes \Sym^n(\cN) \rightarrow \Sym^n(\cM \otimes \cN).
				$$
			\end{enumerate}
		\end{prop}
		
		\begin{prop}
			Let $\cM\in A\Sa$, then $\Sym^n(\cM) \cong \Sat( \mbox{\rm S}^n(M))$ where $\mbox{\rm S}^n$ is the $n^{th}$ symmetric power taken in $A\Gr$ and $\cM=\Sat(M)$.
		\end{prop}
		
		This results holds because of the definition of our tensor product in $A\Sa$, we preserved the monoidal symmetric structure and now each transposition acts by switching tensorands before saturation. More generally, it should be noted that saturating a module corresponds to the sheafification of a presheaf.
		
		We give a more detailed proof of the following proposition first given in \cite{GP}.
		
		\begin{prop}
			For all $M$, $N$ be two modules. We have
			$$
			\Sat(\Sat(M)\otimes \Sat(N)) \cong \Sat(M\otimes_A N).
			$$
		\end{prop}
		\begin{proof}
			Consider the following exact sequence in $A\Gr$ \cite{AZ}:
			$$
			0 \rightarrow \tau(M) \rightarrow M \rightarrow \Sat(M) \rightarrow R^{1}\tau(M) \rightarrow 0
			$$
			where $\tau(M)$ is the largest torsion submodule of $M$. The saturation of $M$, denoted by $\widetilde{M}$, is the maximal essential extension of $M/\tau(M)$ such that $\tM/(M/\tau(M))$ is in $A\Tor$. So we have 
			$$
			0  \rightarrow M/\tau(M) \rightarrow \Sat(M) \rightarrow T \rightarrow 0
			$$
			where $T$ is in $A\Tor$. Applying by $\_\otimes_AN$ we obtain
			$$	
			... \rightarrow \textrm{Tor}_1^A(T,N)  \rightarrow M/\tau(M)\otimes_AN \rightarrow \Sat(M)\otimes_AN \rightarrow T\otimes_AN \rightarrow 0.
			$$
			From the properties of the $\textrm{Tor}$ functor, it is known that $\textrm{Tor}_1^A(T,N) \cong \textrm{Tor}_1^A(N,T)$. Now taking a projective resolution of $N$ and tensoring by $T$ we get a complex of objects in $A\Tor$ since tensor product preserves torsion. Therefore  $\textrm{Tor}_1^A(T,N)$ is in $A\Tor$. The saturation functor is exact and the saturation of torsion objects is the zero object, so we get a short exact sequence
			$$
			0  \rightarrow \Sat(M/\tau(M)\otimes_AN) \rightarrow \Sat(\Sat(M)\otimes_AN) \rightarrow 0.
			$$
			And hence an isomorphism
			$$
			\Sat(M/\tau(M)\otimes_AN) \cong \Sat(\Sat(M)\otimes_AN).
			$$
			Moreover we have the following short exact sequence
			$$
			0  \rightarrow \tau(M) \rightarrow M \rightarrow M/\tau(M) \rightarrow 0.
			$$
			Tensoring on the left by $N$ we get
			$$
			\tau(M)\otimes_AN \rightarrow M\otimes_AN \rightarrow M/\tau(M)\otimes_AN \rightarrow 0.
			$$
			Since $\tau(M)\otimes_AN$ is torsion, applying the saturation functor we obtain
			$$
			0 \rightarrow \Sat(M\otimes_AN) \rightarrow \Sat(M/\tau(M)\otimes_AN) \rightarrow 0.
			$$
			And hence an isomorphism
			$$
			\Sat(M\otimes_AN) \cong \Sat(M/\tau(M)\otimes_AN)
			$$
			So,
			$$
			\Sat(M\otimes_AN) \cong \Sat(\Sat(M)\otimes_AN)
			$$
			To conclude,
			\begin{align*}
			\Sat(M\otimes_AN) &\cong \Sat(\Sat(M)\otimes_AN) \\
			&\cong \Sat(N \otimes_A \Sat(M)) \\
			&\cong \Sat(\Sat(N) \otimes_A \Sat(M)) \\
			&\cong \Sat(\Sat(M) \otimes_A \Sat(N))
			\end{align*}
		\end{proof}
		
		\begin{defn}
			A vector bundle $\cM$ is {\bf ample} if for any coherent sheaf $\cF$ there exists $n_0>0$ such that $$\cF\otimes \Sym^n(\cM)$$ is weighted globally generated for all $n \geqslant n_0$.
		\end{defn}
		
		\begin{prop}
			\begin{enumerate}
				\item For any ample sheaf, there exists a non-negative integer such that any of its higher symmetric power is weighted globally generated.
				\item Any quotient sheaf of an ample sheaf is ample.
			\end{enumerate}
		\end{prop}
		\begin{proof}
			\begin{enumerate}
				\item Suppose $\cM$ is ample, since $\cO$ is a coherent sheaf then there exist a non-negative $n_0$ such that for all $n \geqslant n_0$
				$$
				\cO \otimes \Sym^n(\cM) \cong \Sym^n(\cM)
				$$
				is weighted globally generated.
				\item For a given sheaf $\cF$, the functor $\cF \otimes \_$ is right exact as a composition of a right exact functor and an exact functor. Let $\cM'$ be a quotient of $\cM$, i.e., we have an epimorphism $\cM \twoheadrightarrow \cM'$. Since $\Sym^n$ preserves epimorphisms we have $$\Sym^n(\cM) \twoheadrightarrow \Sym^n(\cM'),$$ so $$\cF \otimes\Sym^n(\cM) \twoheadrightarrow \cF \otimes\Sym^n(\cM')$$ for any coherent sheaf $\cF$. But $\cM$ is ample, so for $n$ sufficiently large $\cF \otimes\Sym^n(\cM)$ is weighted globally generated and since $\cF \otimes\Sym^n(\cM) \twoheadrightarrow \cF \otimes\Sym^n(\cM')$ is an epimorphism then $\cF \otimes\Sym^n(\cM')$ is weighted globally generated. This shows that $\cM'$ is ample.
			\end{enumerate}	
			
		\end{proof}
		
		\begin{prop}
			The finite direct sum of ample sheaves is ample.
		\end{prop}
		\begin{proof}
			The proof is similar to the one given by Hartshorne \cite{Har}.
			We know that
			$$
			\Sym^n(\cM\oplus\cN) = \bigoplus_{p=0}^{n} \Sym^p(\cM)\otimes \Sym^{n-p}(\cN).
			$$
			Write $q=n-p$. It suffices to show that there exists some non-negative integer $n_0$ such that when $p+q \geqslant n_0$ then 
			$$
			\cF\otimes \Sym^p(\cM)\otimes \Sym^p(\cN)
			$$
			is weighted globally generated.
			
			Fix some coherent sheaf $\cF$,
			\begin{enumerate}
				\item $\cM$ is ample so there exists a positive integer $n_1$ such that for all $n \geqslant n_1$, 
				$$
				\Sym^n(\cM)
				$$
				is weighted globally generated.
				\item $\cN$ is ample so there exists a positive $n_2$ such that for all $n \geqslant n_2$, 
				$$
				\cF \otimes \Sym^n(\cN)
				$$
				is weighted globally generated.
				\item For each $r\in\{0,...,n_1-1\}$, the sheaf $\cF \otimes \Sym^r(\cM)$ is coherent. Since $\cN$ is ample, there exists $m_r$ such that for all $n \geqslant m_r$,
				$$
				\cF \otimes \Sym^r(\cM) \otimes \Sym^n(\cN)
				$$
				is weighted globally generated.
				\item For each $s\in\{0,...,n_2-1\}$, the sheaf $\cF \otimes \Sym^s(\cN)$ is coherent. Since $\cM$ is ample, there exists $l_s$ such that for all $n \geqslant l_s$,
				$$
				\cF \otimes \Sym^n(\cM) \otimes \Sym^s(\cN)
				$$
				is weighted globally generated.
			\end{enumerate}
			
			Now take $n_0=\max_{r,s}\{r+m_r,s+l_s\}$, then for any $n\geqslant n_0$ 
			$$
			\cF \otimes \Sym^p(\cM) \otimes \Sym^q(\cN)
			$$
			is weighted globally generated.
			
			Indeed, we have 3 cases,
			\begin{enumerate}[(i)]
				\item Suppose $p <n_1$. Then $p+q\geqslant n_0 \geqslant p+m_p$, so $q\geqslant m_p$ and by 3. we are done.
				\item Suppose $q <n_2$. Then $p+q\geqslant n_0 \geqslant l_q+q$, so $p\geqslant l_q$ and by 4. we are done.
				\item Suppose $p \geqslant n_1$ and $q \geqslant n_2$, so $\Sym^p(\cM)$ and $\cF\otimes \Sym^q(\cN)$ are weighted globally generated and so is their tensor product.
			\end{enumerate}
			We conclude that $\cM \oplus \cN$ is ample.
		\end{proof}
		
		\begin{cor}
			Let $\cM$ and $\cN$ be two sheaves. Then $\cM \oplus \cN$ is ample if and only if $\cM$ and $\cN$ are ample.
		\end{cor}
		\begin{proof}
			We already know that if $\cM$ and $\cN$ are ample then so is their direct sum. Conversely, $\cM$ and $\cN$ are quotient of $\cM\oplus\cN$ which is ample, so are $\cM$ and $\cN$.
		\end{proof}
		
		\begin{cor}
			The tensor product of an ample sheaf and a weighted globally generated sheaf is ample.
		\end{cor}
		\begin{proof}
			Let $\cM$ be an ample sheaf and $\cN$ a weighted globally generated sheaf. So,
			$$
			\bigoplus_{j = 0}^{l-1} \cO(j)^{\oplus s_j} \rightarrow \cN \rightarrow 0.
			$$
			Tensoring by $\cM$,
			$$
			\bigoplus_{j = 0}^{l-1} \cM \otimes \cO(j)^{\oplus s_j} \rightarrow \cM \otimes \cN \rightarrow 0.
			$$
			It suffices to show that $\cM \otimes \cO(j)$ for $j\in\{0,...,l-1\}$ is ample.
			Let $\cF$ to be a coherent sheaf and consider 
			$$
			\cF \otimes \Sym^n(\cM \otimes \cO(j))
			$$
			for $n$ a non-negative integer. It is a quotient of
			$$
			\cF \otimes \Sym^n(\cM) \otimes \Sym^n(\cO(j)) \cong \cF \otimes \Sym^n(\cO(j)) \otimes \Sym^n(\cM).
			$$
			But $\cF \otimes \Sym^n(\cO(j))$ is a coherent sheaf and $\cM$ is ample, so there exists a non-negative integer $n_0$ such that for all $n \geqslant n_0$
			$$
			\cF \otimes \Sym^n(\cO(j)) \otimes \Sym^n(\cM)
			$$
			is weighted globally generated. It follows that all of its quotients are weighted globally generated and in particular $\cF \otimes \Sym^n(\cM \otimes \cO(j))$. Hence, $\cM \otimes \cO(j)$ is ample and the result follows.
		\end{proof}
		
		\begin{lemma}
			The sheaf $\cO(1)$ is ample.
		\end{lemma}
		\begin{proof}
			Let $\cF=\Sat(F)$ be a coherent sheaf. So $F$ is a finitely generated module over $A$ generated by finitely many homogeneous elements $f_0,...,f_c$ of degree $\rho_0,...,\rho_c$ respectively.
			
			Take $n_0=\max\{\rho_0,...,\rho_c\}$, then for each $n \geqslant n_0$ we have
			$$
			n-\rho_i=a_il+r_i
			$$
			where $0\leqslant r_i <l$ by the division algorithm. We claim that the following map is an epimorphism in $A\Sa$,
			$$
			\bigoplus_{j=0}^n \bigoplus_{i=0}^c \cO(r_i) \rightarrow \cF(n)
			$$
			induced by $((0,...,0),...,(0,...,1_i,...,0)_j,...(0,...,0)) \mapsto x_j^{\frac{a_il}{d_j}}f_i$.
			Indeed, to prove the claim we need to show that the cokernel of the map in $A\Gr$ is torsion. So take $f\in F(n)$ homogeneous and assume that $f$ can be written $kf_i$ for some $i\in\{0,...,c\}$. Let $N=\max\left\lbrace  \dfrac{a_il}{d_j}, i\in\{0,...,c\}, j\in\{0,...,n\} \right\rbrace $. Suppose that $h\in A$ is an homogeneous element of degree greater than $N$. So it can be written as $h'x_j^{ \frac{a_il}{d_j}}$ for some $i\in\{0,...,c\}$ and $j\in\{0,...,n\}$. It follows that $hf$ is in the image of the map. Henceforth its cokernel is zero.
		\end{proof}
		
		Since $\cO(1)$ is ample and weighted globally generated, so $\cO(2)$ is ample for any weighted projective stack. However, for $\bP(3,5)$, seen as a variety, we have $\cO(2) \cong \cO(-1)$ which isn't ample.
		
		\begin{theorem}
			The tangent sheaf of any weighted projective stack is ample.
		\end{theorem}
		\begin{proof}
			We have the following short exact sequence \cite{Zho}
			$$
			0 \rightarrow \cO \rightarrow \bigoplus_{j=0}^n \cO(a_j) \rightarrow \cT \rightarrow 0.
			$$
			From the long exact sequence in cohomolgy applied to the above short exact sequence and the fact that $\cO$ and $\cO(a_j)$ are vector bundles, it is evident to see that so is $\cT$.
			Each summand of the central term is ample since $\cO(a_i)=\cO(1)^{\otimes a_i}$ and $\cO(1)$ ample. Moreover, $\cT$ is the quotient of a finite direct sum of ample sheaves. Then, $\cT$ is ample.
		\end{proof}
		
		We obtain the following corollary proved first by Hartshorne
		\begin{cor}[\cite{Har}]
			The tangent sheaf of a standard projective space is ample.
		\end{cor}
		
		A converse of this corollary was proved and provides a characterisation of smooth projective spaces also known as the Hartshorne conjecture and proved by Mori,
		\begin{theorem}[\cite{Mor}]
			The only smooth irreducible projective smooth variety with ample tangent bundle is isomorphic to $\bP^n$ for some $n$.
		\end{theorem}
		
		It is natural to ask whether the conjecture of Hartshorne holds as well extending the class of spaces. We formulate it as follows:
		
		\begin{conj}
			The only smooth irreducible projective stack with an ample tangent bundle is isomorphic to a weighted projective stack.
		\end{conj}	
		
		\bibliography{Reference_4th_year_report}{}
		\bibliographystyle{amsplain}
	\end{document}